\documentclass[12pt]{amsart}
\usepackage{amsmath,amsthm,amscd,amssymb,bbm,graphicx}
\usepackage{floatflt,setspace, wrapfig}
\usepackage{graphics,multicol}
\usepackage{mathrsfs} 
\usepackage{epstopdf}
\usepackage{geometry}
\usepackage{todonotes}

\newtheorem{thm}{Theorem}

\newtheorem{prop}{Proposition}[section]
\newtheorem{lem}{Lemma}[section]

\newtheorem{dfn}{Definition}[section]
\newtheorem{remark}{Remark}
\newtheorem*{cdn}{Condition}
\newtheorem{asm}{Assumption}

\def\br{B^\Omega_r(x)}
\def\bxr{B^X_r(x,s)}

\newcommand{\diam}{\mbox{diam }}

\newcommand{\supp}{\mathrm{supp}}
\newcommand{\ncy}{\xi_{i_0\ldots i_{n-1}}}
\newcommand{\ncys}{\tilde{\xi}_{i_0\ldots i_{n-1}}}

\providecommand{\abs}[1]{\lvert \, #1 \, \rvert}

\newcommand\numberthis{\addtocounter{equation}{1}\tag{\theequation}}

\allowdisplaybreaks

\def\phi{\varphi}
\def\le{\leqslant}
\def\ge{\geqslant}

\def\diam{\text{diam}}

\def\supp{\text{supp}}

\title{Hitting Times Distribution And Extreme Value Law for Flows}
 \date{\today}
\begin{document}

\maketitle
\authors{Maria Jos\'{e} Pacifico\footnote{Partially supported by CNPq, FAPERJ},
Fan Yang\footnote{Partially supported by CAPES.}}
\begin{abstract}
For flows whose return map on a cross section has sufficient mixing property, we show that the hitting time distribution of the flow to balls is exponential in limit. We also establish a link between the extreme value distribution of the flow and its hitting time distribution, generalizing a previous work by Freitas et al in the discrete time case. Finally we show that for maps that can be modeled by Young`s tower with polynomial tail, the extreme value law holds. 

\end{abstract}

\tableofcontents

\section{Introduction}
Hitting and return time statistics for discrete time systems have been studied extensively in the past few decades. The famous Poincar\'e recurrence theorem guarantees that the return time is almost surely finite. More recently, the limiting distribution of the hitting and return time has been studied in various situations. The first result is due to Pitskel~\cite{Pit} for Markov chains and  Hirata~\cite{H93} for Axiom A diffeomorphisms. Later \cite{HSV} provided a frame work to show that the first hitting time distribution is exponential, 
for systems with the short return set controlled and satisfying certain mixing properties. 
Higher order return times have also been studied. In~\cite{CC}, \cite{HW15}  and~\cite{PS} it is shown that for systems that can be modeled by Young's tower, the higher order return times converge in distribution to Poisson law.  

Nowadays 
the study of extreme value distribution in dynamical systems has became in a very active line of research.
The recent published book~\cite{LFF} provides an excellent introduction to this field. Notably Freitas et al.~\cite{FFT10} and~\cite{FFT11} showed a relation between extreme value law and hitting time distribution for proper observations.

However, not much is known about the return time statistics for flows. Rousseau~\cite{R12} established a relation between the recurrence rate and the pointwise dimension of the invariant measure. The same result is proven for both expanding and contracting Lorenz attractors in~\cite{galapacif09} and~\cite{GNP}. 

There is even fewer results on the extreme value law for flows, possible due to the lack of corresponding theory in probability theory itself. Holland et al.~\cite{HNT} established a link between the extreme value theory of a measure preserving map and its suspension flow under a integrable roof function. 

In this work we study the hitting time distribution and extreme value law for the suspension flow over a measure preserving map with a integrable roof function. In Section~\ref{hitting flow} we show that the hitting time distribution of the map being exponential implies the same law for the suspension flow. In Section~\ref{evl flow} we study the relation between the hitting time distribution for the flow and the extreme value law, generalizing the work by Freitas et al.~\cite{FFT11} in the discrete time case. Finally we give a direct proof of the extreme value law for systems with Young towers whose tail is polynomial. 

After uploading this preprint on arXiv, we were made aware that Marklof in~\cite{M16} showed a result similar to Theorem 1. He considered target sets on cross sections instead of balls and showed that the Poisson distribution for the hitting time of the discrete time systems implies the same law for the semi-flow. 

\section{Definitions and Main Theorems}

Let $R$ be a measurable map on the metric space $\Omega$ with metric $d_\Omega$. Let $\mu_\Omega$ be a probability measure which is invariant under $R$. We consider the suspension flow $X_t$ on the space $X = \Omega \times [0,\infty)/\sim$ with roof function $r: \Omega \to \mathbb{R}^+$where the equivalent relation $\sim$ is given by 
$$
\left(x,r(x)\right) \sim \left(R(x),0\right).
$$
On the space $X$ we use the box metric, i.e.\ $d_X((x,t_1), (y,t_2)) = \max\{d_\Omega(x,y), \abs{t_2 - t_1} \} $. In this paper we will always assume that the roof function is integrable, i.e.\ $\int r\,d\mu_\Omega < \infty$. This allows us to define the flow-invariant probability measure $\mu_X$ for the suspension flow $X_t$ as follow:
$$
d\mu_X:= \frac{d\mu_\Omega \times dt}{\int r\,\mu_\Omega}
$$

\subsection{Hitting times distribution for suspension flows}\label{hitting flow}

Given $A \subset \Omega$ we define the first hitting time and higher order hitting times for the map $R$ as usual:
$$
 \tau^R_A(x) = \tau^{1,R}_A(x) = \min\{k>0: R^k(x) \in A\}
$$
and for $m>0$,
$$
\tau^{m+1,R}_A(x) = \min\{k>\tau^{m,R}_A(x): R^k(x) \in A\}.
$$
By the famous Poincar\'e Recurrence Theorem, $\tau^R_A$ is almost surely finite on $A$. Furthermore by the Kac's Lemma, we have
$$
\int_A \tau^R_A\, d\mu_\Omega = 1 
$$
provided that the measure $\mu_\Omega$ is ergodic.

We say that the first hitting time distribution of $R$ is exponential, if 
$$
\mu_\Omega\left(\tau^R_{B^\Omega_r(x)} > \frac{t}{\mu_\Omega(B^\Omega_r(x))} \right) \to e^{-t}
$$
as $r\to 0$. Here $B^\Omega_r(x)$ is the ball at $x$ with radius $r$.

We say that the higher order hitting times distribution is Poisson, if
$$
\mu_\Omega\left(\tau^{m,R}_{B^\Omega_r(x)} \le \frac{t}{\mu_\Omega(B^\Omega_r(x))} <  \tau^{m+1,R}_{B^\Omega_r(x)} \right) \to e^{-t} \frac{t^m}{m!}.
$$

Next we define the hitting times for the flow $X_t$.

\begin{dfn}\label{dfn_hitting}
Let $X_t$ be a flow on some space $M$. Given $A\subset M$ we define the  exit time $E_A(x) = \inf\{t\ge 0: X_t(x) \notin A\}$. 
The first hitting time is defined as 
$$
\tau^X_A(x) = \tau^{1,X}_A(x) = \inf\{t>E_A(x): X_t(x) \in A \}.
$$
Similarly the higher order hitting times are defined as
$$
\tau^{m+1,X}_A(x) = \inf\{t>\tau^{m,X}_A(x) +  E_A(X_{\tau^{m,X}_A(x)}(x)): X_t(x) \in A \}.
$$
\end{dfn}

The first theorem establishes the relation between the hitting times for the flow and the hitting times of the return map $R$.

\begin{thm}\label{thm_hitting}
Let $X_t$ be the suspension flow over a map $R: \Omega \to \Omega$ with an integrable roof function $r$. Assume the probability measure $\mu_\Omega$ is invariant and ergodic under $R$ and the measure $\mu_X$ is the induced measure for the flow as before. Let  $B^X_r(x,t_0)$  be the closed r-ball under the metric $d_X$  centered at $(x,t_0)$. Then the following hold:\\
1. If the first hitting time distribution of $R$ is exponential, then
$$
\mu_X\left(\tau^X_{\bxr}> \frac{t\cdot 2r}{\mu_X(\bxr))}\right) \to e^{-t}\quad \mbox{as}\,\, r\to 0.
$$
2. If the higher order hitting times distribution of $R$ is Poisson, then
$$
\mu_X\left(\tau^{m,X}_{\bxr} \le \frac{t\cdot 2r}{\mu_X(\bxr)} <  \tau^{m+1,X}_{\bxr} \right) \to e^{-t} \frac{t^m}{m!}\quad \mbox{as}\,\, r\to 0.
$$
\end{thm}

\begin{remark}
The purpose of the  exit time is to deal with the case when the point $(x,t)$ is already in the ball. In this case we have $E_{\bxr} \le 2r$.
\end{remark}





\subsection{Extreme value law for flows}\label{evl flow}
In this section we establish  the relation between the hitting times distribution and the extreme value law for the flow. Let $X_t$ be as before, preserving a probability measure $\mu$. We follow the notation in~\cite{FFT11} and assume that the observable $\varphi$ has the form
$$
\varphi(x) = g\left(\frac{\mu(B_{d(x,z)}(z))}{2d(x,z)}\right)
$$
where $z\in X$ is a given point and $g$ is a function from $X$ to $\mathbb{R}\cup\{+\infty\}$ with the following properties: $g$ is strictly decreasing in a neighborhood of $0$; $0$ is a global maximum for $g$; $g$ satisfies one of the following three properties:

\noindent Type 1. There exist some strictly positive function $p$ such that
$$
g^{-1}\left(g(\frac{1}{t})+yp(g(\frac{1}{t}))\right)=(1+\varepsilon_t)\frac{e^{-y}}{t}.
$$
with $\epsilon_t \to 0$ as $t \to \infty$ for all $y \in \mathbb{R}$.

\noindent Type 2. $g(0) = +\infty$ and there exist $\beta>0$ such that for all $y>0$ it holds
$$
g^{-1}\left(g(\frac{1}{t})y\right) = (1+\varepsilon_t)\frac{y^{-\beta}}{t}.
$$

\noindent Type 3. $g(0) = D<+\infty$ and there exist $\gamma>0$ such that for all $y>0$ it holds
$$
g^{-1}\left(D-g(\frac{1}{t})y\right) =  (1+\varepsilon_t) g^{-1}\left(D-g(\frac{1}{t})\right)\frac{y^{\gamma}}{t}.
$$

Examples of functions satisfying the three types are $g_1(x) = -\log x$, $g_2(x) = x^{-\frac{1}{\beta}}$ and $g_3(x) = D-x^{\frac{1}{\gamma}}$   respectively .

Let $G$ be a distribution function. We say that the flow $X_t$ has hitting times distribution $G$ to balls, if
$$
\lim\limits_{r \to 0}\mu\left(\tau^X_{B_r} > \frac{t\cdot 2r}{\mu(B_r)} \right) = G(t)
$$
where $\tau^X_{B_r}$ is defined as in Definition~\ref{dfn_hitting}. Note that the $2r$ on the numerator is due to the fact that the flow direction does not affect the hitting time.

Like in the discrete case we put $Y_0 = \varphi$ and $Y_t = \varphi\circ X_t$. Let $M_t = \sup\{Y_s: 0 \le s \le t\}$ and $\tau_1(y) = e^{-y} $ for $y \in \mathbb{R}$, $\tau_2(y) = y^{-\beta}$ for $y>0$ and $\tau_3(y) = (-y)^\gamma$ for $y \le 0$. We have the following result.

\begin{thm}\label{thm_evl_flow}
Let $z\in \supp(\mu)$ be a point such that $h_z(r) := \mu(B_r(z))$ is continuous in $r$. Assume we have hitting times distribution $G$ to balls centered at $z$, then we have extreme value law for $M_t$ of the form $H(y) = G(\tau_i(y))$ to the observable $g_i$, $i =1,2,3$. 

In particular, if the $G(t) = e^{-t}$ we get the Gumbel law, Fr\'echet law and Weibull law like in the discrete case respectively for type 1, 2 and 3.
\end{thm}

\subsection{Extreme value law for Young's tower with polynomial tail}

Extreme value law has been studied extensively for the last decade.  In~\cite{HNT} it is shown that maps modeled by Young's tower with exponential tail satisfies the extreme value law. For Young's tower with polynomial tail, Haydn, Wassilewska~\cite{HW15} and Pene, Saussol\cite{PS} showed that such maps has exponentially distributed return time. This, together with the Theorem 1 in~\cite{FFT11} showed that maps modeled by Young's tower with polynomial tail also satisfy the extreme value law. Here we give a direct proof of the extreme value law by directly verifying the condition $D_2$ and $D'$ given in~\cite{FF08}. For more details see Section~\ref{thm_evl_tower}. 

We will make the following assumptions on the measurer $\mu$:
\begin{asm}\label{Annulus} There exist $\delta>1$ such that for $r>0$ small enough we have
$$
\mu(B_{r+r^\delta}(x)\setminus B_{r}(x)) = o(\mu(B_r(x)))
$$
for almost every $x$.
\end{asm}

\begin{asm}\label{Dimension}There exist $C,C'>0$, $d_1>d_0>0$ such that for $\mu$-almost every $x$
$$
C'r^{d_1}<\mu(B_r(x)) < Cr^{d_0}.
$$
\end{asm}
For a fixed $z\in M$ we put 
$$
h_z(r) = \mu(B_r(z)). 
$$
Then $h$ is a non-increasing function.  Assumption~\ref{Annulus} implies that for almost every $z$, $h_z(r)$ is continuous from the right for $r$ small. Notice that both assumptions are satisfied by measures that are absolutely continuous w.r.t. Lebesgue with a bounded density. In this case we can choose $d_0$ and $d_1$ close to the dimension of $M.$

\begin{thm}\label{thm_evl_tower}
Assume that $T$ can be modeled by a Young's tower with polynomial tail, i.e., \ $\mu(\{R>n\}) \le C n^{-p}$ for some $p>8$. Also assume that the measure $\mu$ satisfies the regularity assumptions~\ref{Annulus} and~\ref{Dimension}.

Then $T$ satisfies the extreme value law.
\end{thm}

All the proofs can be found in Section~\ref{proof}.

\section{Proof of Theorems}\label{proof}

\subsection{Proof of Theorem~\ref{thm_hitting} }
First we establish the relation between the hitting time of the suspension flow and the return map $R$. For every $(y_0,t_0) \in X$ let $T(y_0,t_0) = r(y_0)-t_0$, i.e., \ $T$ is the minimal time under the flow for orbit beginning at $(y_0,t_0)$ to hit $\Omega \times \{0\}$. Then for every $A \subset \Omega$ and $s>0$ we have 
$$
\tau^X_{X_{[-s,s]}(A)}(y_0,t_0) =T(y_0,t_0) + \sum_{i=0}^{\tau^R_A(y)-1} r(R^i(y))   -s.
$$
Here $y=X_{T(y_0,t_0)}(y_0,t_0) = (R(y_0),0)$ and $X_{[-s,s]}(A)$ is the flow box as before. Notice that this hold true even for points $(y_0,t_0)$ contained in $X_{[-s,s]}(A)$. Similarly for higher order hitting times we have 
\begin{equation}
\tau^{m,X}_{X_{[-s,s]}(A)}(y_0,t_0) =T(y_0,t_0) + \sum_{i=0}^{\tau^{m,R}_A(y)-1} r(R^i(y)) -s.
\end{equation}\label{hitting_flow}

Now we fix some $x \in A$ and take $A = \br$. Since we use the box metric we have $X_{[s-r,s+r]}(A)=\bxr$

By the Birkhoff's ergodic theorem, 
$$
\frac{1}{n}\sum_{i=0}^{n-1}r(R^i(z)) \to\int_{\Omega}r(z)\,d\mu_\Omega.
$$
For $z$ in a set of full measure which me denote by $G$. Since $\mu_\Omega$ is ergodic, for almost every $z$ and every $m \ge 1$ we have $\tau^{m,R}_A(y) \to \infty$ as $r\to 0$. Taking another full measure subset of $G$ if necessary, we have for $y\in G$,
\begin{align*}
\tau^{m,X}_{\bxr}(y_0,t_0) =& \sum_{i=0}^{\tau^{m,R}_{\br}(y)-1} r(R^i(y)) + T(y_0,t_0)+s-r\\
= & c(y,r,m) \cdot \tau^{m,R}_{\br}(y)\int_{\Omega}r(z)\,d\mu_\Omega + T(y_0,t_0)+s-r
\end{align*}
with $c(y,r,m) \to 1$ as $r\to 0$

Now we can prove Theorem~\ref{thm_hitting}. Notice that item 1 in Theorem~\ref{thm_hitting} is a special case of item 2 if we define $\tau^{0,R} = \tau^{0,X} = 0$. To simplify notation, write 
$T = T(y_0,t_0)$ and 
$$
 \mathbb{E}(r) = \int_{\Omega}r(z)\,d\mu_\Omega, \quad \tau^R(m)=\tau^{m,R}_{\br}(y) \quad \mbox{and}\quad \tau^X(m)=\tau^{m,X}_{\bxr}(y_0,t_0) .
$$
Then we get
\begin{align*}
&\left\{\tau^X(m)\le \frac{t\cdot 2r}{\mu_X(\bxr)} < \tau^X(m+1)\right\}\\
=&\left\{\sum_{i=0}^{\tau^R(m)-1} r(R^i(y)) + T+s-r \le \frac{t\cdot 2r}{\mu_X(\bxr)}<\sum_{i=0}^{\tau^R(m)-1} r(R^i(y)) + T+s-r\right\}\\
=&\left\{c(y,r,m) \tau^R(m)\mathbb{E}(r)  \le \frac{t\cdot 2r}{\mu_X(\bxr)}-T-s+r<c(y,r,m+1) \tau^R(m+1)\mathbb{E}(r)\right\}\\
=& \left\{\tau^R(m) \le \frac{t\cdot 2r}{c(y,r,m)\mathbb{E}(r)  \cdot\mu_X(\bxr)} - \frac{T+s-r}{c(y,r,m)\mathbb{E}(r)},\right.\\&\hspace*{5mm} \left.\tau^R(m+1) > \frac{t\cdot 2r}{c(y,r,m+1)\mathbb{E}(r)  \cdot\mu_X(\bxr)} - \frac{T+s-r}{c(y,r,m+1)\mathbb{E}(r) }.\right\}\end{align*}

Since $r$ is integrable, $T=T(y_0,t_0) = r(y_0)-t_0$ is almost surely finite. Thus if we fix a $T_0$ big enough and define $G_{T_0} = \{y\in\Omega: r(y) \le T_0\}$, this set has measure close to one. Since $c(y,r,m)$ and $c(y,r,m+1)$ both converge to $1$ as $r \to 0$ for $y\in G$, we have $\frac{T+s-r}{c(y,r,m)} = \mathcal{O}(1)$ on $G \cap G_{T_0}$; the same hold for $m+1$. Dropping these two terms from the previous set will result in an error which can be estimated as 
\begin{align*}
\mu_X(error_1) =& \mu_X\left((y_0,t_0):y_0\in G\cap G_{T_0} \text{ and } \tau^R(m) \le \frac{T+s-r}{c(y,r,m)}\right)\\
 \le& \mu_\Omega\left(y_0\in G\cap G_{T_0}: \tau^R(m) \le \frac{T+s-r}{c(y,r,m)}\right)\\
 \le& \mu_\Omega\left( \tau^R(m) \le \frac{T+s-r}{c(y,r,m)}\right)\\
 \le& \mu_\Omega\left(\br \cup R^{-1}(\br) \cup \cdots \cup R^{-\frac{T+s-r}{c(y,r,m)}}(\br) \right)\\
 \le&  \frac{T+s-r}{c(y,r,m)} \cdot\mu_\Omega(\br) \to 0
 \end{align*}
as $r\to 0$. Similarly the error cause by changing $c(y,r,m+1)$ to $c(y,r,m)$ can be estimated as
\begin{align*}
\mu_X(error_2) \le& \mu_X\bigg((y_0,t_0):y_0\in G\cap G_{T_0} \text{ and } \\
&\left. \frac{t\cdot 2r}{c(y,r,m)\mathbb{E}(r)  \mu_X(\bxr)}<\tau^R(1)\le \frac{t\cdot 2r}{c(y,r,m+1)\mathbb{E}(r) \mu_X(\bxr)}\right)\\
 \le& \mu_\Omega\left( \frac{t\cdot 2r}{c(y,r,m)\mathbb{E}(r)  \mu_X(\bxr)}<\tau^R(1)\le \frac{t\cdot 2r}{c(y,r,m+1)\mathbb{E}(r) \mu_X(\bxr)}\right)\\
 \le&  \frac{t\cdot 2r}{\mathbb{E}(r)\mu_X(\bxr)}\left|\frac{1}{c(y,r,m)} - \frac{1}{c(y,r,m+1)} \right| \mu_\Omega(\br).
 \end{align*}
Notice that 
\begin{align*}
\frac{ 2r}{\mathbb{E}(r)\mu_X(\bxr)} =& \frac{ 2r}{\mathbb{E}(r)\mu_X(X_{[-r,r]}(\br))}\\
 =& \frac{ 2r}{\mathbb{E}(r)\mu_\Omega(\br)\cdot \frac{2r}{\mathbb{E}(r)}}\\
 =& \frac{1}{\mu_\Omega(\br)},
\end{align*}
then we have 
$$
\mu_X(error_2) \le\left|\frac{1}{c(y,r,m)} - \frac{1}{c(y,r,m+1)} \right| \to 0
$$
since both $c(y,r,m)$ and $c(y,r,m+1)$ converge to $1$ as $r$ converges to $0$.

Hence we are left to show that 
\begin{align*}
\mu_X\left((y_0,t_0):y_0\in G\cap G_{T_0} \text{ and }\tau^R(m) \le \frac{t\cdot 2r}{c(y,r,m)\mathbb{E}(r)  \cdot\mu_X(\bxr)}<\tau^R(m+1)\right)
\end{align*}
converges to the Poisson distribution. Dropping the restriction on $G\cap G_{T_0}$ is negligible since the measure of this set can be made arbitrarily close to one by taking a bigger $T_0$. Since
\begin{align*}
&\mu_X\left(\tau^R(m) \le \frac{t\cdot 2r}{c(y,r,m)\mathbb{E}(r)  \cdot\mu_X(\bxr)}<\tau^R(m+1)\right)\\
=&\mu_\Omega\left(\tau^R(m) \le \frac{t}{c(y,r,m)  \cdot\mu_\Omega(\br)}<\tau^R(m+1)\right)
\end{align*}
and $c(y,r,m) \to 1$ the result follows.

\subsection{Proof of Theorem~\ref{thm_evl_flow}}
We only prove Theorem~\ref{thm_evl_flow} for $i=1$, function $g$ with $g^{-1}\left(g(\frac{1}{t})+yp(g(\frac{1}{t}))\right)=(1+\varepsilon_t)\frac{e^{-y}}{t}.$ The other two cases follow similarly.

For given $z \in Z$, define $l(y) = \inf\{r>0:\frac{\mu(B_r(z))}{2r} = y\}$. It is easy to see that if $h_z(r) = \frac{\mu(B_r(z))}{2r}$ is continuous at r then $\frac{\mu(B_{l(y)}(z))}{2l(y)} = y.$

Set $\tau = e^{-y}$ for $y \in \mathbb{R}$. Define
$$
u_t(y) = g\left(\frac{1}{t}\right) + p\left(g\left(\frac{1}{t}\right)\right) y.
$$
Recall that $Y_t = \varphi\circ X_t$ and $\varphi(y) = g\left(\mu(B_{d(y,z)}(z))\right)$ we have the following:

\begin{align*}
\mu(Y_0 > u_t) =& \mu\{x: g\left(\frac{\mu(B_{d(x,z)}(z))}{2d(x,z)}\right) > u_t\}\\
=& \mu\{x: \frac{\mu(B_{d(x,z)}(z))}{2d(x,z)} < g^{-1}(u_t)\}\\
=& \mu\{x:  \frac{\mu(B_{d(x,z)}(z))}{2d(x,z)} < \frac{\mu\left(B_{l(g^{-1}(u_t))}(z)\right)}{2l(g^{-1}(u_t))} \}\\
=& \mu\{x: d(x,z) < l(g^{-1}(u_t))\} \numberthis \label{11}\\
=& \mu(B_{l(g^{-1}(u_t))}(z))\\
=& g^{-1}(u_t) \cdot 2l(g^{-1}(u_t)) = (1+\varepsilon_t)\frac{\tau \cdot 2l(g^{-1}(u_t))}{t}.
\end{align*}
Here we use the assumption that $g$ is strictly decreasing in a neighborhood of $0$. \ref{11} is due to the fact that the measure $\mu$ is Lebesgue along the flow direction. The last two lines also give
$$
t = \frac{(1+\varepsilon_t)\tau\cdot 2l(g^{-1}(u_t))}{\mu(B_{l(g^{-1}(u_t))}(z))}.
$$

Similarly for every $t>0$ we have 

\begin{align*}
\{M_t < u_t\} =& \bigcap_{s=0}^t \{X_s < u_t\}\\
=& \bigcap_{s=0}^t \{ g\left(\frac{\mu(B_{d(X_s(x),z)}(z))}{2d(X_s(x),z)}\right) < u_t\}\\
=& \bigcap_{s=0}^t \{\frac{ \mu(B_{d(X_s(x),z)}(z))}{2d(X_s(x),z)} > g^{-1}(u_t)\}\\
=& \bigcap_{s=0}^t \{ \frac{\mu(B_{d(X_s(x),z)}(z))}{2d(X_s(x),z)} >  \frac{\mu(B_{l(g^{-1}(u_t))}(z))}{2l(g^{-1}(u_t))} \}\\
=& \bigcap_{s=0}^t \{d(X_s(x),z) > l(g^{-1}(u_t)) \}\\
=& \{\tau^X_{B_{l(g^{-1}(u_t))}(z)}(x) > t\}.
\end{align*}
Hence
\begin{align*}
\mu(M_t < u_t) =& \mu(\tau^X_{B_{l(g^{-1}(u_t))}(z)}(x) > t)\\
=&\mu\left(\tau^X_{B_{l(g^{-1}(u_t))}(z)}(x) > \frac{(1+\varepsilon_t)\tau\cdot 2l(g^{-1}(u_t))}{\mu(B_{l(g^{-1}(u_t))}(z))}\right).
\end{align*}

To simplify notations we set $T_t = \frac{\tau \cdot 2l(g^{-1}(u_t))}{\mu(B_{l(g^{-1}(u_t))}(z))}$ and $T'_t=\frac{(1+\varepsilon_t)\tau\cdot 2l(g^{-1}(u_t))}{\mu(B_{l(g^{-1}(u_t))}(z))}$. Assuming without loss of generality that $\varepsilon_t>0$, we have 

\begin{align*}
&\left|\mu(\tau^X_{B_{l(g^{-1}(u_t))}(z)}(x) > T_t) - \mu(\tau^X_{B_{l(g^{-1}(u_t))}(z)}(x) > T'_t) \right|\\
=& \mu\left(T_t< \tau^X_{B_{l(g^{-1}(u_t))}(z)}(x)\le T'_t\right)\\
\le& \mu\left(X_{[T',T'_t]}(B_{l(g^{-1}(u_t))}(z))\right)\\
\le& C(T'_t - T_t)\frac{\mu(B_{l(g^{-1}(u_t))}(z))}{2l(g^{-1}(u_t))}\\
=& \varepsilon_t \tau.
\end{align*}
Here  we use again the fact that $\mu$ is Lebesgue along  the flow direction.
It follows that 
\begin{align*}
\left|\mu(M_t < u_t) - G(\tau)\right| = & \left|\mu\left(\tau^X_{B_{l(g^{-1}(u_t))}(z)}(x) > T'_t \right) -  G(\tau)\right|\\
\le& \left|\mu\left(\tau^X_{B_{l(g^{-1}(u_t))}(z)}(x) > T_t \right) -G(\tau) \right|
\\&+\left|\mu\left(\tau^X_{B_{l(g^{-1}(u_t))}(z)}(x) > T_t \right) -\mu\left(\tau^X_{B_{l(g^{-1}(u_t))}(z)}(x) > T'_t \right)\right|\\
\le& \left|\mu\left(\tau^X_{B_{l(g^{-1}(u_t))}(z)}(x) > T_t \right) -G(\tau) \right| + \varepsilon_t \tau.
\end{align*}
Since $X_t$ has hitting times distribution $G$ and $\varepsilon_t \to 0$, the theorem follows.

\subsection{Proof of Theorem~\ref{thm_evl_tower}}
Let $\{Y_n\}$ be a stochastic process and define
$$
M_n = \max\{Y_0, Y_1, \ldots, Y_n\}
$$
and
$$
M_{l,n} = \max\{Y_l, Y_{l+1}, \ldots, Y_{l+n}\}.
$$

It is enough to verify conditions $D_2(u_n)$ and $D'(u_n)$ given in ~\cite{FF08}. 
\begin{cdn} $D_2(u_n)$-
We say condition  $D_2(u_n)$ holds if for any integers $l,t$ and $n$
$$
|\mu(Y_0>u_n, M_{t,l}<u_n) - \mu(Y_0>u_n)\mu(M_l<u_n)| \le \gamma(n,t)
$$
where $\gamma(n,t)$ is a non-increasing sequence in $t$ for every $n$ and satisfies $\gamma(n,t_n) = o(\frac{1}{n})$ for some sequence $t_n = o(n)$, $t_n \to \infty$.
\end{cdn}
\begin{cdn} $D'(u_n)$- We say condition $D'(u_n)$ holds if 
$$
\lim\limits_{k\to\infty}\limsup_{n} n\cdot\sum_{j=1}^{[n/k]}\mu(Y_0>u_n, Y_j>u_n ) =0.
$$
\end{cdn}

As in the previous section we only prove the theorem for Type 1 observables, namely 
$$
\varphi(x) = g(\mu(B_{d(x,z)}(z)))
$$
with $g$ satisfying
$$
g^{-1}\left(g(\frac{1}{n})+yp\left(g(\frac{1}{n})\right)\right) = (1+\varepsilon_n)\frac{e^{-y}}{n}.
$$
\subsubsection{Proof of $D_2(u_n)$}
First we show $D_2(u_n)$. Following the same proof as Lemma~3.1 in~\cite{GHN} we get 
\begin{lem}\label{decay}
Let $\Phi$ be Lipschitz and $\Psi$ be the indicator function of any measurable set. Then for every $j>0$ we have
$$
\left| \int \Phi \Psi\circ T^j\, d\mu -\int\Phi\,d\mu \int \Psi\circ T^j\, d\mu\right| \le \mathcal{O}(1)\left( \|\Phi\|_\infty \tau_1^{[j/2]} + \|\Phi\|_{Lip} \cdot[j/2]^{-p}\right)
$$
for some $0<\tau_1<1$.
\end{lem}

Put $u_n(y) = g\left(\frac{1}{n}\right) + yp\left(g\left(\frac{1}{n}\right)\right)$ then
$$
\{Y_0>u_n\} = B_{l(g^{-1}(u_n))}(z);
$$
here $l(y) = \inf\{r>0: \mu(B_r(z))=y\}$ as before. Recall that $h_z(r)$ defined as $h_z(r) = \mu(B_r(z))$ is continuous from the right, we get that 
$$
\mu(B_{l(y)}(z)) = y.
$$
By Assumption~\ref{Dimension} we get
$$
C y^{1/d_0}\le l(y) \le C' y^{1/d_1}
$$
for some constant $C$ and $C'$. To simplify notations we write $r_n = l(g^{-1}(u_n))$ and omit $z$. We have
\begin{equation}\label{measure_ball}
\mu(B_{r_n}(z)) = g^{-1}(u_n) = (1+\varepsilon_n)\frac{e^{-y}}{n}
\end{equation}
and
\begin{equation}\label{radius}
 C n^{-1/d_0} \le r_n \le C' n^{-1/d_1}.
\end{equation}

Next we approximate the indicator function of $\{Y_0>u_n\} =B_{r_n}$ by a Lipschitz function $\Phi$ in the following way: $\Phi =1$ on $ B_{r_n}$; $\Phi = 0$ outside $ B_{r_n+r_n^\delta }$ and $\Phi$ decays linearly to $0$ on the annulus $ B_{r_n+r_n^\delta } \setminus B_{r_n} $. Here $\delta$ is the constant given in Assumption~\ref{Annulus}. Therefore $\Phi$ has Lipschitz norm bounded by $1/r_n^\delta$. Denote by $\Psi_{t,l}$ the indicator function of $M_{t,l}$. By the previous lemma we have 
\begin{align*}
&\left|\mu(Y_0>u_n, M_{t,l}<u_n) - \mu(Y_0>u_n)\mu(M_l<u_n)\right| \\
=& \left|\int 1_{B_{r_n}} \Psi_{[t/2],l}\circ T^{t-[t/2]}\,d\mu - \int 1_{B_{r_n}}\,d\mu \int\Psi_{0,l}\,d\mu \right|\\
\le&  \mathcal{O}(1)\left( \|\Phi\|_\infty \tau_1^{[t/2]} + \|\Phi\|_{Lip} \cdot[t/2]^{-p}\right) + \left|\int\left(\Phi - 1_{B_{r_n}}\right)  \Psi_{[t/2],l}\circ T^{t-[t/2]}\,d\mu\right|\\
&+\left|\int\left(\Phi - 1_{B_{r_n}}\right)\,d\mu \cdot \int \Psi_{[t/2],l}\,d\mu\right|\\
\le& \mathcal{O}(1)\left( \|\Phi\|_\infty \tau_1^{[t/2]} + \|\Phi\|_{Lip} \cdot[t/2]^{-p} \right) + 2\mu(B_{r_n+r_n^\delta } \setminus B_{r_n})\\
\le&\mathcal{O}(1)\left( \tau_1^{[t/2]} + n^{\delta/d_0}\cdot[t/2]^{-p}  \right)+ 2\mu(B_{r_n+r_n^\delta } \setminus B_{r_n}).
\end{align*}

Let $\gamma(n,t)=\tau_1^{[t/2]} + n^{\delta/d_0}\cdot[t/2]^{-p}  + 2\mu(B_{r_n+r_n^\delta } \setminus B_{r_n}) $. Fix some $\gamma \in (0,1) $ and put $t_n = n^\gamma$, we get
$$
n\gamma(n,t_n) \le \mathcal{O}(1)\left(n^{1+\delta/d_0-p\gamma}+ 2n\mu(B_{r_n+r_n^\delta } \setminus B_{r_n}) \right).
$$
Recall that $\mu(B_{r_n}) = \mathcal{O}(1/n)$ by~(\ref{measure_ball}), we get that
$$
n\mu(B_{r_n+r_n^\delta } \setminus B_{r_n}) \to 0\quad \mbox{as}\quad n\to \infty.
$$
In order that $n\gamma(n,t_n)\to 0$  as $n\to \infty$ we need 
\begin{equation}\label{r0}
1+\delta/d_0-p\gamma<0
\end{equation}
which can be achieved by taking $p> 1+\delta/d_0$ and $\gamma$ close to 1.

Hence taking $p>2+4/d_0$ and $2\gamma>d_1$ we get condition $D_2(u_n)$.

Notice that the proof does not require the tower structure.

\subsubsection{Proof of $D'(u_n)$}
Here we prove another version of $D'(u_n)$ which, together with $D_2(u_n)$, implies the extreme value law. To be more precise we will show that for some choice of $\varepsilon>0$, we have
$$
n \cdot \sum_{j=1}^{n^\varepsilon} \mu(Y_0>u_n, Y_j>u_n)\to 0 \quad \mbox{as}\quad n\to\infty.
$$
See Section 2.1 of~\cite{GHN} for more details.

Recall that $\{Y_j>u_n\} = T^{-j}B_{r_n}(z)$, hence
$$
n \cdot \sum_{j=1}^{n^\varepsilon} \mu(Y_0>u_n, Y_j>u_n) = n \cdot \sum_{j=1}^{n^\varepsilon}  \mu(B_{r_n} \cap T^{-j}B_{r_n}).
$$
We split the sum into two parts: \\
Case 1. $\log n < j < n^\varepsilon$.\\
Case 2. $1 < j < \log n$.

To deal with Case 1 we follow the estimate in Section 4.7 of~\cite{HW15} and get
\begin{align*}
 \mu(B_{r_n} \cap T^{-j}B_{r_n}) \le C\mu(B_{r_n})\left(\sqrt{j}\cdot\Omega(s) +s\cdot r_n^{d_0} + s\left(\sigma^{-j/s}\right)^{d_0} \right)
\end{align*}
for some $\sigma<1$, where $\Omega(s)$ is defined as
$$
\Omega(s) = \sqrt{\sum_{i:R_i>s}R_i m(\Lambda_i)}.
$$
Notice that 
$$
\mu(B_{r_n}) = \mu(B_{l(g^{-1}(u_n))}(z)) = g^{-1}(u_n) =(1+\varepsilon_n)\frac{e^{-y}}{n}.
$$
Fix some constant $0<a<1$ and put $s=j^a$. By Lemma 4.2 of~\cite{HW15}
\begin{align*}
 \mu(B_{r_n} \cap T^{-j}B_{r_n}) \le C\frac{1}{n}\left( j^{1/2-a(p-1)/2}+j^a n^{-d_0/d_1} + j^a\sigma^{-d_0j^{1-a}}\right).
\end{align*}
Now we sum over $j$ and get 
\begin{align*}
n\sum_{j=\log n}^{n^\varepsilon}\mu(B_{r_n} \cap T^{-j}B_{r_n}) \le& n\sum_{j=\log n}^{n^\varepsilon}C\frac{1}{n}\left( j^{1/2-a(p-1)/2}+j^a n^{-d_0/d_1} + j^a\sigma^{-d_0j^{1-a}}\right)\\
\le& C\left( (\log n)^{3/2-a(p-1)/2} + n^{a\varepsilon+\varepsilon -d_0/d_1}\right).
\end{align*}
Thus we need\\
\begin{equation}\label{r2}
\begin{cases}
3/2-a(p-1)/2<0\\
a\varepsilon+\varepsilon -d_0/d_1<0.
\end{cases}
\end{equation}
Now, $a<1$ requires
$$
\begin{cases}
p>4\\
\varepsilon<\frac{d_0}{2d_1}.
\end{cases}
$$
This proves Case 1.
\vspace{0.1cm}

For Case 2 we put $V_r = \{x: B_r(x) \cap T^{-j}B_r(x) \neq \emptyset \text{ for some } 1\le j < \log n\}$ and
\begin{equation}\label{level_set}
N_r(j) = \{x: B_r(x) \cap T^{-j}B_r(x) \neq \emptyset\};
\end{equation}
then $V_r  = \bigcup_{j=1}^{[\log n]}N_r(j)$. The next lemma is a modification of Proposition 5.1 in~\cite{HW15}.

\begin{lem}\label{main_lemma}
There exist $C>0, \delta>0$ such that 
$$
\mu(V_{r_n}) \le C (\log n)^{-\delta}.
$$
\end{lem}

Before proving te lemma we introduce some notations. We denote the base of the tower by $\Lambda = \bigcup_{j=0}^\infty \Lambda_j$ with $R|\Lambda_j = R_j$. Let $\mathcal{P} = \{\Lambda_j\}$ be the partition on the base. Define $\hat{T}: \Lambda \to \Lambda$ as 
$$
\hat{T}x = T^{R_j}x \text{\hspace{2mm} for all } x \in \Lambda_j .
$$
Let $m$ be the SRB measure on $\Lambda$ the $\mu$ the measure on the whole tower with the form
$$
\mu(A) = \sum_{i=0}^{\infty}\sum_{n=0}^{R_i-1} m(T^{-n}A\cap\Lambda_i).
$$

As in~\cite{HW15}, given an index $(i_0, \ldots, i_{n-1})$ we define the $n$-cylinder $\ncy$ by
$$
\ncy = \Lambda_{i_0}\cap \hat{T}^{-1}\Lambda_{i_1} \cap \cdots \cap \hat{T}^{-(n-1)}\Lambda_{i_{n-1}}.
$$
Sometimes we also write $\tau=(i_0, \ldots, i_{n-1})$ and $\xi_\tau$ instead of $\ncy$.
For a given $s>0$ we define 
$$
\ncys = \begin{cases}
\ncy &\mbox{if } R_{i_j}\le s \mbox{ for all } 0 \le j\le n-1\\
\emptyset &\mbox{otherwise.}
\end{cases}
$$
In other words, $\{\ncys\}$ are cylinders that only visit lower levels of the tower. Also put 
$$
\hat{\Lambda}_i = \bigcup_{\ncys \subset \Lambda_i} \ncys,
$$
then $\Lambda_i \setminus \hat{\Lambda}_i$ are the points in $\Lambda_i$ that visit higher levels of the tower at some iterates. We also put $A = \sup\{\|DT\|\}$.

In what follows we will often omit the index of $r_n$ and simply write $r$.
\begin{proof}
The idea is similar to Section 5 of~\cite{HW15}. We fix some $b>0$ and set $s=(\log n)^{1/4}$. Split $V_r$ into two parts:
$$
V^1_r =  \bigcup_{j=1}^{[b\log n]-1}N_r(j), \hspace{5mm}
V^2_r =  \bigcup_{j=[b\log n]}^{[\log n]}N_r(j),
$$
where $N_r(j)$ is given by (\ref{level_set}).
We estimate $V^2_r$ first. 
\begin{align*}
\mu(N^r(j)) =& \sum_{i=1}^{\infty}\sum_{m=0}^{R_i-1} m(T^{-m}N_r(j)\cap\Lambda_i)\\
=&  \sum_{i=1}^{\infty}\sum_{m=0}^{R_i-1} m(T^{-m}N_r(j)\cap\tilde{\Lambda}_i) +  \sum_{i=1}^{\infty}\sum_{m=0}^{R_i-1} m(T^{-m}N_r(j)\cap\Lambda\setminus\tilde{\Lambda}_i).
\end{align*}

The second term is estimated by the tail of the tower and by Lemma 4.7 in \cite{HW15}:
\begin{align*}
\sum_{i=1}^{\infty}\sum_{m=0}^{R_i-1} m(T^{-m}N_r(j)\cap\Lambda\setminus\tilde{\Lambda}_i)
\le&\sum_{i=1}^{\infty}\sum_{m=0}^{R_i-1} m(\Lambda\setminus\tilde{\Lambda}_i)\\
\le& \sum_{i=1}^{\infty}R_i m(\Lambda\setminus\tilde{\Lambda}_i)\\
\le& Cj (\log n)^{-p/4}\\
\le& C (\log n)^{-\frac{p-4}{4}}.
\end{align*}

To deal with the first term, we choose indices $\tau = (i_0, \ldots, i_{l})$ such that $R^{(l)}$ defined as $R^{(l)} = \sum_{k=0}^{l-1}R_{i_k}$ satisfies
$$
R^{(l)} \le j+m \le R^{(l+1)}.
$$
Since we are in $\tilde{\Lambda}_i$, all $R_{i_k}$ are less than $s = (\log n)^{1/4}$.
Thus
$$
m(T^{-m}N_r(j)\cap\tilde{\Lambda}_i) = \sum_{\xi_\tau \subset \Lambda_i} m(T^{-m}N_r(j)\cap\xi_\tau).
$$
Now we use distortion to get
\begin{align*}
m(T^{-m}N_r(j)\cap\xi_\tau) =& \frac{m(T^{-m}N_r(j)\cap\xi_\tau)}{m(\xi_\tau)}m(\xi_\tau)\\
\le& C \frac{m(\hat{T}^{l+1}T^{-m}N_r(j)\cap\xi_\tau)}{m(\hat{T}^{l+1}\xi_\tau)}m(\xi_\tau)\\
\le& Cm(\hat{T}^{l+1}T^{-m}N_r(j)\cap\xi_\tau)m(\xi_\tau).
\end{align*}
Next we estimate the measure of $\hat{T}^{l+1}T^{-m}N_r(j)\cap\xi_\tau$ by its diameter. 
Set $B = R^{(l+1)} - j-m$, then  $B<(\log n)^{1/4}$ and 
$$
\hat{T}^{l+1} = T^{R^{(l+1)}} = T^{j+m+B}.
$$
We take points $x,y\in N_r(j)$ such that $T^{-m}x, T^{-m}y \in T^{-m}N_r(j)\cap\xi_\tau$. We have $d(T^jx,x)<r, d(T^jy,y)<r$ and thus
\begin{align*}
d(\hat{T}^{l+1}T^{-m}x,\hat{T}^{l+1}T^{-m}y) =& d(T^{j+B}x,T^{j+B}y)\\
\le& A^B d(T^jx, T^jy)\\
\le& A^B\left(d(T^jx,x) + d(x,y) + d(T^jy,y)\right)\\
\le& A^B(2r+ A^md(T^{-m}x,T^{-m}y)).
\end{align*}
Since $T^{-m}x, T^{-m}y \in \xi_\tau$ we get $d(T^{-m}x,T^{-m}y) \le \diam(\xi_\tau) < \lambda^l$.

To find a lower bound for  $l$, recall that $j+m \le R^{(l+1)}$ and $R_{i_k}< (\log n)^{1/4}$, we get
$$
j\le R^{(l+1)} \le (l+1) (\log n)^{1/4},
$$
so 
$$
l > j (\log n)^{1/4} > b (\log n)^{3/4}.
$$
As a result,
\begin{align*}
d(\hat{T}^{l+1}T^{-m}x,\hat{T}^{l+1}T^{-m}y) 
\le& A^B2r+ A^{B+m} \lambda^{b (\log n)^{3/4}}\\
\le& CA^{(\log n)^{1/4}} n^{-1/d_1}+ CA^{2(\log n)^{1/4}} \lambda^{b (\log n)^{3/4}}\\
\le& Ce^{-C'(\log n)^{3/4}}.
\end{align*}
Together with Assumption~\ref{Dimension} we get
$$
m(T^{-m}N_r(j)\cap\xi_\tau) \le Ce^{-C'd_0(\log n)^{3/4}} m(\xi_\tau).
$$
Collecting all $\xi_\tau$ and summing over $m$ gives
\begin{align*}
 \sum_{i=1}^{\infty}\sum_{m=0}^{R_i-1} m(T^{-m}N_r(j)\cap\tilde{\Lambda}_i)
 =& \sum_{i=1}^{\infty}\sum_{m=0}^{R_i-1}\sum_{\xi_\tau \subset \Lambda_i} m(T^{-m}N_r(j)\cap\xi_\tau)\\
 \le& \sum_{i=1}^{\infty}\sum_{m=0}^{R_i-1}\sum_{\xi_\tau \subset \Lambda_i}Ce^{-C'd_0(\log n)^{3/4}} m(\xi_\tau)\\
 \le& Ce^{-C'd_0(\log n)^{3/4}} .
\end{align*}
The overall estimate for $V_r^2$ is
\begin{align*}
\mu(V_r^2) =&\sum_{j=[b\log n]}^{[\log n]}\mu(N_r(j))\\
\le& \sum_{j=[b\log n]}^{[\log n]} Ce^{-C'd_0(\log n)^{3/4}} +C'' (\log n)^{-\frac{p-4}{4}}\\
\le& C(\log n)^{-\frac{p-8}{4}}.
\end{align*}

For $V_r^1$, we use Lemma B.3 in~\cite{CC} and put
$$
s_p = 2^p\frac{A^{j2^p}-1}{A^j-1}.
$$ 
Then $N_r(j) \subset N_{s_pr}(2^pj)$ for any $p>1$. Let $p(j) = [\log_2(b\log n) - \log_2j]+1$, we get
$$
\bigcup_{j=1}^{[b\log n]}N_r(j) \subset \bigcup_{j=1}^{[b\log n]} N_{s_{p(j)}r}(2^{p(j)}j) \subset \bigcup_{j' = [b\log n]}^{[2b\log n]}N_{r'}(j')
$$
with $r' = s_p(j)r$ and $j' = 2^{p(j)}j$. Notice that 
$$
r' = s_{p(j)}r \le b\log n \cdot A^{b\log n} r \le n^{b\log A-1/d_1}.
$$
The same estimate as before gives
\begin{align*}
\mu(N_{r'}(j'))\le&\left( A^B2r'+ A^{B+m} \lambda^{b (\log n)^{3/4}}\right)^{d_0} + C (\log n)^{-\frac{p-4}{4}}\\
\le& C''e^{(\log n)^{1/4}}n^{b\log A-1/d_1} +C' e^{c (\log n)^{3/4}}+C (\log n)^{-\frac{p-4}{4}}\\
\le& C (\log n)^{-\frac{p-4}{4}}
\end{align*}
if $b\log A-1/d_1<0$, i.e., $b < 1/(d_1 \log A)$. As a result 
$$
\mu(V_r^1) \le \sum_{j' = [b\log n]}^{[2b\log n]} \mu(N_{r'}(j')) \le C (\log n)^{-\frac{p-8}{4}}.
$$
Therefore we get 
$$
\mu(V_r) \le \mu(V_r^1) +\mu(V_r^2) \le  C (\log n)^{-\delta}
$$
with $\delta = \frac{p-8}{4}$.
This finishes the proof of Lemma \ref{main_lemma}.
\end{proof}

To finish the proof of Theorem~\ref{thm_evl_tower} we use the Maximal function technique by Collet~\cite{C01}. See also~\cite{GHN}.
 This allows us to carry Lemma~\ref{main_lemma} over to neighborhoods of generic points.

Fix some $0<\zeta<1, \rho>0$, we define the set 
$$
F_k = \left\{\mu(B_{r_{\exp(k^\zeta)}} \cap V_{r_{\exp(k^\zeta)}}) \ge \mu(B_{r_{\exp(k^\zeta)}}) \cdot k^{-\rho}\right\}.
$$
$x \in F_k$ means that 
$$
\frac{\mu(B_{r_{\exp(k^\zeta)}} \cap V_{r_{\exp(k^\zeta)}})}{ \mu(B_{r_{\exp(k^\zeta)}}) } \ge k^{-\rho}
$$
i.e., the portion of points of $V_{r_{\exp(k^\zeta)}}$ in the $r_{\exp(k^\zeta)}$-neighborhood at $x$ is at least $ k^{-\rho}$.
Next we define the maximal function as 
$$
M_r(x) = \sup_{s>0}\frac{1}{\mu(B_s(x))}\int_{B_s(x)}1_{V_r}(y)\,d\mu(y).
$$
We have
$$
F_k \subset \{M_{r_{\exp(k^\zeta)}} \ge k^{-\rho}\}.
$$
By the theorem of Hardy and Littlewood (see e.g.\ Theorem 2.19 of~\cite{M95}) we get 
\begin{align*}
\mu(F_k) \le& \mu(M_{r_{\exp(k^\zeta)}} \ge k^{-\rho}) \le \frac{\mu(1_{V_{r_{\exp(k^\zeta)}}(y)})}{k^{-\rho}}\\
\le& k^{-(\zeta\delta -\rho)}.
\end{align*}
If $\zeta\delta -\rho>1$ we get $\sum_k\mu(F_k) < \infty$, thus there exist $N(x)$ for almost every $x$ such that $x \notin F_k$ for all $k>N(x)$. For every $n$, recall that $C/n \le r_n \le C'/n$. We choose $k$ such that 
$$
\exp(k^\zeta)\le n < \exp((k+1)^\zeta),
$$
we have $r_{\exp((k+1)^\zeta)}\le r_n \le r_{\exp(k^\zeta)}$. As a result
\begin{align*}
B_{r_n}\cap T^{-j}B_{r_n} \subset& B_{r_{\exp(k^\zeta)}}\cap T^{-j} B_{r_{\exp(k^\zeta)}}\\
\subset& B_{r_{\exp(k^\zeta)}}\cap V_{r_{\exp(k^\zeta)}}
\end{align*}
for every $j < \log n$. Therefore
\begin{align*} 
n \cdot \sum_{j=1}^{\log n}  \mu(B_{r_n} \cap T^{-j}B_{r_n}) \le& n \cdot \sum_{j=1}^{\log n} \mu( B_{r_{\exp(k^\zeta)}}\cap V_{r_{\exp(k^\zeta)}})\\
\le & \exp((k+1)^\zeta) \cdot (k+1)^\zeta \mu(B_{r_{\exp(k^\zeta)}}) \cdot k^{-\rho}\\
\le & C\, \frac{ \exp((k+1)^\zeta)}{\exp(k^\zeta)}\cdot k^{\zeta -\rho} \to 0,
\end{align*}
provided that $\zeta - \rho <0$. To summarize we need 
$$
\begin{cases}
0<\zeta<\rho\\
\zeta<1\\
\zeta\delta -\rho>1.
\end{cases}
$$
Such parameters exist if $\delta = \frac{p-8}{4}>2$, i.e.\ $p>16$. All together finishes the proof of Theorem~\ref{thm_evl_tower}.
\section{Examples}
\subsection{Lorenz attractors}
The Lorenz flow \cite{Lo63} is one of the key examples in the theory of dynamical systems due to the chaotic nature of its dynamics, its robustness and its connection with hydrodynamical systems. The Lorenz attractor, a {\em{strange attractor}} with a characteristic butterfly shape, has extremely rich dynamical properties which have been studied from a variety of viewpoints:
topological, geometric and statistical, see \cite{ArPa10}. Part of the reason for the richness of the Lorenz flow is the fact that it has an equilibrium, i.e., a fixed point, which is accumulated by regular orbits (orbits through points where the corresponding vector field does not vanish) which prevents the flow from being uniformly hyperbolic. Indeed it is one of the motivating examples in the study of non-uniformly hyperbolic dynamical systems \cite{MPP99}. It is also robust in the sense that nearby flows also possess strange attractors with similar properties. The Lorenz equations can be studied using geometric models of the Lorenz flow, see \cite{ABS, GW79}. 

The classical geometric Lorenz flow is expanding. This corresponds to the Lyapunov exponents at the origin, $\lambda_s$ and $\lambda_u$, the stable and unstable exponents respectively, having $\lambda_u + \lambda_s > 0$.  The geometric Lorenz attractor has a global cross section $\Sigma \subset R^3$ , and a first return map $F$ defined on $\Sigma \setminus \Gamma$, where $\Gamma$ is a one dimensional line contained in the stable manifold at the origin, that preserves a one-dimensional foliation which is contracted under the action of $F$. It is possible to study the dynamics of a geometric Lorenz attractor flow through a $1$-dimensional map $f$ obtained from the quotient  though the leaves of this contracting foliation. We refer to \cite{galapacif09} for a didactic exposition of this construction.


It is shown in~\cite{APPV} that $F$ has a SRB measure $\mu$ whose conditional measures on the unstable manifolds are absolutely continuous w.r.t.\ Lebesgue with bounded density. In \cite{galapacif09} it is shown that $\mu$ has exponential decay of correlation with Lipschitz functions. To show that $F$ has exponentially distributed first return time we need the following elementary result, which is similar to Assumption~\ref{Annulus}:

\begin{prop}\label{regularity_lorenz}
There exist $C>0$ such that 
$$
\mu(B_{r+\epsilon}(x)\setminus B_r(x)) \le C(r\epsilon)^{1/2}
$$
for every $0<\epsilon \ll r$.
\end{prop} 
\begin{proof}
Let $\mu_{\gamma_u}$ be the conditional measure of $\mu$ on the unstable leaf $\gamma_u$. Denote by $\nu$ the transversal measure and $A_{r,\epsilon} = B_{r+\epsilon}(x)\setminus B_r(x)$. We get
$$
\mu(A_{r,\epsilon}) = \int \int_{\gamma_u} \mathbbm{1}_{A_{r,\epsilon}\cap\gamma_u}\,d\mu_{\gamma_u}d\nu(\gamma_u)
$$
Easy calculation shows that for each $\gamma_u$, $A_{r,\epsilon}\cap\gamma_u$ consists of at most two segments, each of which has length bounded by $c(r\epsilon)^{1/2}$ for some constant $c$. Assume that the densities of $\mu_{\gamma_u}$ are bounded by $D$ for all $\gamma_u$, the result follows with $C = c\cdot D.$
\end{proof}

\begin{remark}
Proposition~\ref{regularity_lorenz} implies Assumption~\ref{Annulus} by taking $\epsilon = r^\delta$ with $\delta> 2d_H-1$, where $d_H$ is the Hausdorff dimension of $\mu$.
\end{remark}

Next we use Theorem 2.1 of~\cite{HSV} to conclude that the first hitting time distribution of $F$ is exponential. The proof involves approximating the indicator function of $B_r$ by a Lipschitz function as in the proof of condition $D_2$  and a Lebesgue density point argument. For more details see \cite{RSV}.

It then follows from Theorem~\ref{thm_hitting} that the first hitting time distribution of the Lorenz flow is exponential. By Theorem~\ref{thm_evl_flow} we conclude that the classical Lorenz flow satisfies the extreme value law.

\subsection{Suspension flow over $C^2$ diffeomorphisms }

We consider  a dynamical system on a compact manifold $M$ where $f :
M \to M$ is a $C^2$ diffeomorphism with an attractor $A$. Suppose the system can be modeled by
a Young tower whose return time function decays polynomially with degree $p > 16$. Take an integrable function $r:m \to \mathbb{R}^+$ and let $X$ be the suspension flow over $f$ with roof function $r$.

Let $\mu$ be the SRB measure associated with the tower. Assumption~\ref{Dimension} is satisfied by taking $d_1>d>d_0$ where $d$ is the Hausdorff dimension of $\mu$. Assumption~\ref{Annulus} can be deduced from Proposition 6.1 in~\cite{HW15}. It follows from Theorem~\ref{thm_evl_tower} that $f$ satisfies the extreme value law. By Theorem 2.6 in~\cite{HNT} the flow $X$ also has the same extreme value distribution.

On the other hand, Theorem~\ref{thm_hitting} together with Theorem 3 in~\cite{HW15} shows that $X$ has exponentially distributed first hitting time. Applying Theorem~\ref{thm_evl_flow} we get again extreme value law for $X$.






%

\vspace{1cm}
\noindent
{\em  M. J. Pacifico and F. Yang}:
Instituto de Matem\'atica,
Universidade Federal do Rio de Janeiro,
C. P. 68.530, CEP 21.945-970,
Rio de Janeiro, RJ, Brazil.
\vspace{0.2cm}

E-mail: pacifico@im.ufrj.br $\&$ mjpacifico@gmail.com

E-mail: fyang@im.ufrj.br $\&$ fizbanyang@gmail.com

\end{document}